\documentclass[12pt]{article}
\usepackage[margin=1in]{geometry} 
\usepackage{amsmath,amsthm,amssymb,amsfonts}
\usepackage{xcolor}
\usepackage{blkarray}
\usepackage{graphicx}
\usepackage{float}
\usepackage[shortlabels]{enumitem}
\usepackage{comment}
\usepackage[ruled,vlined]{algorithm2e}

\newcommand{\Z}{\mathbb{Z}}

\newtheorem{theorem}{Theorem}[]
\newtheorem*{theorem*}{Theorem}

\newtheorem{corollary}{Corollary}[theorem]
\newtheorem{lemma}[theorem]{Lemma}

\newtheorem{observation}[theorem]{Observation}
\newtheorem{conjecture}[theorem]{Conjecture}
\newtheorem{proposition}[theorem]{Proposition}

\newtheorem{definition}[theorem]{Definition}

\usepackage{hyperref}
\hypersetup{
    colorlinks,
    citecolor=black,
    filecolor=black,
    linkcolor=blue,
    urlcolor=black
}

\newcommand{\kat}[1]{{\textcolor{blue}{#1}}}
\usepackage[normalem]{ulem}

\usepackage{hyperref}
\hypersetup{
    colorlinks=true,
    linkcolor=black,
    }

\title{Nowhere-zero 8-flows in 3-edge-connected signed graphs}
\author{Matt DeVos \qquad Kathryn Nurse \quad Robert \v S\'amal}
\author{
   Matt DeVos\thanks{
     Email: {\tt mdevos@sfu.ca}. 
     Supported by an NSERC Discovery Grant (Canada).}
 \and
   Kathryn Nurse\thanks{
  Email: {\tt knurse@sfu.ca}.
  Supported by an NSERC Vanier Scholarship (Canada).}
  \and
    Robert \v S\'amal\thanks{
    Email: {\tt samal@iuuk.mff.cuni.cz}.     
    Supported by grant 22-17398S of the Czech Science Foundation. 
    }
}

\begin{document}
\maketitle

\begin{abstract}
    {
    
    In 1983, A. Bouchet extended W.T. Tutte's notion of nowhere-zero flows to signed graphs, and conjectured that every flow-admissible signed graph has a nowhere-zero 6-flow. In this paper we prove that every flow-admissible signed graph that is 3-edge-connected has a nowhere-zero 8-flow. This is a continuation of a previous paper where we proved the same conclusion under stronger assumptions.}
\end{abstract}




\section{Introduction}

Motivated by proper colourings of planar graphs, W.T. Tutte introduced the notion of a nowhere-zero flow in the 1950's. He made three profound flow conjectures: Tutte's 3-flow conjecture (1972) is a generalization of Grotzsch's theorem that triangle-free planar graph are 3-colourable. It was proved to hold for 6-edge-connected graphs (2013) by Lov\'asz, Thomassen, Wu, and Zhang \cite{LTWZ}. Tutte's 4-flow conjecture (1966) is a generalization of the 4-colour theorem. It was shown (1997-2016), and widely accepted, although not completely published, that the 4-flow conjecture holds for cubic graphs \cite{RobertsonNeil1997TEC, EDWARDS201666}.     Tutte's 5-flow conjecture (1954) is generally considered the most important of these conjectures. 

\begin{conjecture}[Tutte]\label{5fl}
    Every flow-admissible graph has a nowhere-zero 5-flow.
\end{conjecture}

Seymour proved that every flow-admissible graph has a nowhere-zero 6-flow \cite{Seymour6flow} in 1981. Two years later, Bouchet noticed that Tutte's flow-colouring duality extends to general surfaces. Interestingly, however, this requires flows on signed graphs. Bouchet made the following conjecture, which parallels Tutte's 5-flow conjecture in the broader setting of signed graphs \cite{Bouchet}.

\begin{conjecture}[Bouchet]\label{bouchetConj}
    Every flow-admissible signed graph has a nowhere-zero 6-flow.
\end{conjecture}

Like the 5-flow conjecture, Bouchet's 6-flow conjecture is best-possible due to (a signed version of) the Petersen graph \cite{Bouchet}. Bouchet proved that his conjecture holds with 6 replaced by 216. This has been further improved by numerous authors and the best current result, published in 2021, is that every flow-admissible signed graph has a nowhere-zero 11-flow \cite{DLLZZ}.

\medskip

The present paper is a continuation of \cite{devos2023nowherezero}, where we proved that every flow-admissible signed graph that is cyclically 5-edge-connected and cubic has a nowhere-zero $8$-flow. Here, we prove the same conclusion under weaker assumptions.

\begin{theorem}\label{mainthm}
Every flow-admissible, 3-edge-connected signed graph has a nowhere-zero 8-flow.
\end{theorem}

Throughout, we consider (signed) graphs which are finite, but may have multiple edges or loops. Our methods here are based on those in \cite{devos2023nowherezero}, and we use notation defined there and in \cite{Diestel3rdEdition}. In particular, expanded definitions about signed graphs such as signature, switching, balance, etc., and flows on those graphs can be found in \cite{devos2023nowherezero}. We briefly review them now. 

A \emph {signed graph} is a graph $G$ equipped with a bipartition of its edges $E(G)$ into \emph{negative} and \emph{positive}, which we call a \emph{signature}. An \emph{oriented} signed graph, also called a bidirected graph, is a signed graph equipped with exactly two \emph{directions} for every edge, one assigned to each of its ends (both directions of a loop-edge are assigned to its end). 
This orientation must respect the signature of $G$: A positive edge is oriented so that one of its directions is towards and the other is away. A negative edge is oriented so that both of its directions are the same (both towards or both away). These directions are represented by arrowheads in Figure \ref{fig:orienting signed edges}.

\begin{figure}[h]
    \centering
    \includegraphics[width=0.5\linewidth]{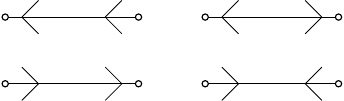}
    \caption{Orientations of signed edges. Positive edges on left, negative edges on right.}
    \label{fig:orienting signed edges}
\end{figure}

We allow two operations on an oriented signed graph $G$: reversing edges, and switching at vertices. To \emph{reverse} an edge $e\in E(G)$ is to reverse both of the directions of $e$ (from away to towards and/or from towards to away). This has no effect on the signature of $G$. To \emph{switch} at a vertex $v\in V(G)$ means to reverse every edge-direction that is assigned to $v$. This changes the signature of $G$, so that positive non-loop edges incident to $v$ become negative, and negative non-loop edges incident to $v$ become positive.

Let $G$ be an oriented signed graph, and $A$ an abelian group. For $v \in V(G)$, denote by $\partial^+(v)$ the set of edges directed away from $v$ and denote by $\partial^-(v)$ the set of edges directed toward $v$. With apologies, we include a negative loop-edge incident to $v$ twice in one of these sets (depending on its orientation).
A function $f: E(G) \to A$ is an $A$-\emph{flow} whenever $\partial f(v) := \sum_{e \in \partial ^+(v)}f(e) - \sum_{e \in \partial^-(v)}f(e) = 0$ at every vertex $v$. An integer-flow $f$ is called a \emph{nowhere-zero $k$-flow} whenever the range of $f$ is in $\{\pm1, \pm2, \dots, \pm(k-1)\}$. 

It follows from the above discussion that reversing an edge preserves the existence of a flow in a signed graph $G$: when $G$ has a flow $f$, if we reverse an edge $e$ and replace $f(e)$ with its inverse, then $f$ remains a flow. 

It also follows that switching at a vertex preserves the existence of a flow. Hence, we say two signed graphs $G$ and $G'$ are \emph{equivalent} if $G'$ can be obtained from $G$ by a sequence of switching operations. 

Therefore, to demonstrate the existence of a nowhere-zero $k$-flow in a signed graph $G$ it suffices to demonstrate a nowhere-zero $k$ flow in any orientation of $G$ and with any equivalent signature of $G$.

A signed graph is \emph{$k$-unbalanced} whenever each of its equivalent signatures has at least $k$ negative edges. A $1$-unbalanced signed graph is simply \emph{unbalanced}.
A signed graph is \emph{flow-admissible} whenever it admits a nowhere-zero integer flow. A 2-edge-connected, unbalanced signed graph is flow-admissible if and only if it is 2-unbalanced \cite{Bouchet}.


\medskip

This paper proceeds as follows. In Section 2, we reduce proving Theorem \ref{mainthm} to proving a result about cubic graphs with an additional connectivity property. After this, Sections 3 to 6 parallel those sections in \cite{devos2023nowherezero}, where we build our list of cycles and establish a preflow. Because of our weaker assumptions, the cycle selection process here is considerably more involved. Once our list of cycles is built, however, with slight modification, we are able to use the machinery already established in \cite{devos2023nowherezero}.
In Section \ref{3ec usable cycle}, we show that we can always find a usable cycle for our process -- this time the reduced connectivity assumptions force us to look in a leaf-block of our graph. In Section \ref{3ecunbaltheta}, we establish Lemma \ref{theBeast} to ensure the number of vertices in our subdivided graph is even (so that we may 
find a perfect matching in our graph). Our weaker connectivity assumptions force us to allow an exception graph which we call \emph{fish}. In Section~\ref{csa}, we describe our cycle-selection process -- this time via an algorithm. Here we are forced to allow weaker connectivity with the positive cycles than we did in \cite{devos2023nowherezero}.
Finally, using the machinery established in the last three sections of \cite{devos2023nowherezero}, we put everything together to prove our main result in Section~\ref{3ecmain}.

\section{Reducing the main theorem}\label{3ecs1}

In this section we show that the main theorem can be reduced to proving a result for cubic graphs with an additional connectivity property.  To achieve this, we will consider a counterexample graph $G$ for which the parameter $3|E(G)| - 4|V(G)|$ is minimum.  Let us note here that for every graph $G$ with minimum degree at least 3 we have 
\[ 3|E(G)| - 4 |V(G)| = \tfrac{3}{2} \sum_{v \in V(G)} \mathrm{deg}(v) - 4 |V(G)| \ge \tfrac{1}{2}|V(G)|, \]
with equality when $G$ is cubic.  In particular, this quantity must always be positive for every 3-edge-connected graph.  



\medskip

We require a proposition about uncontracting edges from \cite{doi:10.1137/23M1615218}.
Suppose $G$ is a signed graph, $v \in V(G)$ with $\deg(v) \geq 4$, and $e,f$ are distinct edges incident to $v$ whose other ends are $v_e, v_f$ respectively (we are allowing $e,f$ to be loops). We \emph{uncontract at $v$ with $\{e,f\}$}\index{uncontract} to get a new signed graph by adding new vertex $v'$, changing the ends of $e,f$ to be $v_ev'$ and $v_fv'$ respectively, and adding a positive edge $vv'$. See Figure \ref{fig:uncontract}. 

\begin{figure}[htb]
    \centering
    \includegraphics[height=2.2cm]{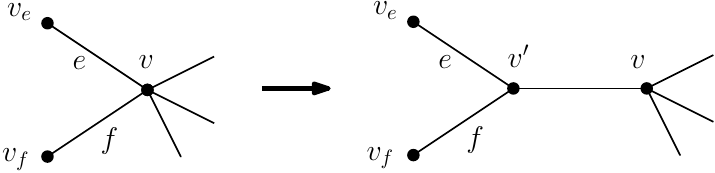}
    \caption{Uncontracting at $v$ with $\{e,f\}$.}
    \label{fig:uncontract}
\end{figure}

\begin{proposition}[\cite{doi:10.1137/23M1615218}]\label{prop:can uncontract and keep 2-unbal 3ec}
Let $G$ be a 2-unbalanced, 3-edge-connected signed graph on at least two vertices. Let $v \in V(G)$ have degree at least four, and $e \in E(G)$ be incident to $v$. Then there is an edge $e'$ incident to $v$ so that the graph $G'$, obtained from $G$ by uncontracting at $v$ with $\{e,e'\}$, is 2-unbalanced and 3-edge-connected.
\end{proposition}

\begin{lemma}\label{getcubic}
    If $G$ is a counterexample to Theorem \ref{mainthm} for which $3|E(G)| - 4|V((G)|$ is minimum, then $G$ is cubic.
\end{lemma}

\begin{proof}
    Suppose (for a contradiction) that $G$ is not cubic and choose a vertex $v \in V(G)$ with degree greater than 3.  By Proposition \ref{prop:can uncontract and keep 2-unbal 3ec} we may uncontract a positive edge $e$ at the vertex $v$ so that the resulting signed graph $G'$ is both flow-admissible and 3-edge-connected.  Since $G$ is a minimum counterexample and $G'$ has one more edge and vertex, the theorem applies to $G'$ and we may choose a nowhere-zero 8-flow for $G'$.  Now contracting $e$ gives a nowhere-zero 8-flow of $G$, a contradiction.
\end{proof}

In preparation for the next reduction, we prove two lemmas concerning flows in ordinary (not signed) graphs with prespecified values on some edges.

\begin{lemma}
    Let $G$ be an oriented 2-edge-connected graph and let $v \in V(G)$ be incident with exactly three edges $e_1, e_2, e_3$ all directed away from $v$.  Let $\Gamma$ be an abelian group with $|\Gamma| \ge 6$ and let $g_1, g_2, g_3 \in \Gamma \setminus \{0\}$ satisfy $g_1 + g_2 + g_3 = 0$.  Then there exists a nowhere-zero $\Gamma$-flow $\phi$ of $G$ with $\phi(e_i) = g_i$ for $i = 1,2,3.$
\end{lemma}

\begin{proof}
    This follows from Tutte's deletion-contraction formula\footnote{For graph $G$ and $e \in E(G)$, the number of nowhere-zero $\Gamma$-flows in $G$ is equal to the number of nowhere-zero $\Gamma$-flows in $G$ contract $e$ ($G / e$) minus the number of nowhere-zero $\Gamma$-flows in $G$ delete $e$ ($G - e$). See, for example, \cite{Diestel3rdEdition} Section 6.3, page 145.} for counting flows.  Define
    \[ \mathcal{T} = \{ (h_1, h_2, h_3) \in ( \Gamma \setminus \{0\} )^3 \mid h_1 + h_2 + h_3 = 0 \}. \]
    If $G'$ is a minor of $G$ obtained by deleting and contracting edges other than $e_1, e_2, e_3$ and $T = (h_1, h_2, h_3) \in \mathcal{T}$ we define $\Gamma_T(G')$ to be the number of nowhere-zero $\Gamma$-flows $\phi$ of $G'$ that satisfy $\phi(e_i) = h_i$.\\
    
    \noindent\emph{Claim}:  
    $\Gamma_T(G') = \Gamma_{T'}(G')$
    holds for every $T,T' \in \mathcal{T}$. 
    
    We prove this claim by induction on $|E(G')|$.  As a base case, when $G'$ has just the edges $e_1, e_2, e_3$ the resulting graph has exactly one suitable nowhere-zero flow if $e_1, e_2, e_3$ are parallel, and none otherwise.  For the inductive step, choose an edge $f \in E(G) \setminus \{e_1,e_2,e_3\}$ and apply the deletion-contraction formula together with the inductive hypothesis as follows:
    \[ \Gamma_T(G') = \Gamma_T(G' / f) - \Gamma_T(G' - f ) = \Gamma_{T'}(G' / f) - \Gamma_{T'}(G' - f) = \Gamma_{T'}(G') \]
    This completes the proof of claim.\\

    By Seymour's 6-flow theorem \cite{Seymour6flow}, we must have $\Gamma_T(G) \neq 0$ for some $T \in \mathcal{T}$ and then by the above argument, this must hold true for every triple in $\mathcal{T}$.
\end{proof}

\begin{lemma}\label{prespecify_vertex_integer}
    Let $G$ be an oriented 2-edge-connected graph and and let $v \in V(G)$ be incident with exactly three edges $e_1, e_2, e_3$ all directed away from $v$.  Let $k \ge 6$ and let $g_1, g_2, g_3 \in \{ -(k-1), \ldots k-1 \} \setminus \{0\}$ satisfy $g_1 + g_2 + g_3 = 0$.  Then there exists a nowhere-zero $k$-flow $\phi$ of $G$ satisfying $\phi(e_i) = g_i$ for $i=1,2,3$.
\end{lemma}

\begin{proof}
    By the previous lemma, we may choose a nowhere-zero flow $\psi : E(G) \rightarrow \mathbb{Z}_k$
so that $\psi(e_i) = g_i + k \mathbb{Z}$ for $i =1,2,3$.  By Tutte's Theorem\footnote{This is well-known lore. See, for example, the proof of \cite[Theorem 21.3]{BondyMurty2008}.}, we may choose a $k$-flow $\phi : E(G) \rightarrow \mathbb{Z}$ so that $\phi(e) + k \mathbb{Z} = \psi(e)$ holds for every $e \in E(G)$.  Now for $1 \le i \le 3$ the edge $e_i$ either has the desired value: $\phi(e_i) = g_i$, or has $\phi(e_i) = g_i \pm k$.  Since $\sum_{i=1}^3 \phi(e_i) = 0 = \sum_{i=1}^3 g_i$ either all three edges have the desired value and we are done, or by rearranging indices we may assume $\phi(e_1) = g_1$, $\phi(e_2) = g_2 + k$ and $\phi(e_3) = g_3 - k$.  Note that this implies $\phi(e_2) > 0$ and $\phi(e_3) < 0$.  Modify $\phi$ and the orientation of $G$ as follows:  for every edge $f \in E(G) \setminus \{e_1,e_2,e_3\}$ with $\phi(f) < 0$, reverse $f$ and multiply $\phi(f)$ by $-1$.  Note that after this operation $\phi$ is still a nowhere-zero $k$-flow and $\phi$ is positive on all edges apart from $e_1, e_2, e_3$. 

Let $e_i$ have ends $v$ and $u_i$ for $i=1,2,3$.  We claim that there is a directed path from $u_2$ to $u_3$ in the graph $G - v$.  Suppose (for a contradiction) that no such path exists.  Then there is a partition of $V(G) \setminus \{v\}$ into $\{A,B\}$ with $u_2 \in A$ and $u_3 \in B$ so that there are no edges with tail in $A$ and head in $B$.  If $u_1 \in A$ then add $v$ to the set $A$, otherwise $u_1 \in B$ and we add $v$ to $B$.  After this modification, all edges directed from $B$ to $A$ have positive flow in $\phi$ and all edges directed from $A$ to $B$ have negative flow, contradicting the cut rule for flows.  Therefore, we may choose a directed path $P$ from $u_2$ to $u_3$ in $G-v$.  Modify the flow $\phi$ by subtracting $k$ on all edges in $E(P) \cup \{e_2\}$ and adding $k$ on $e_3$.  This results in a new flow that satisfies the lemma.
\end{proof}

For a (signed) graph $G=(V,E)$ and a set $X \subseteq V(G)$ we use $d(X)$ to denote the number of edges with one end in $X$ and one end in $V\setminus X$.
To \emph{identify} $x,y \in V(G)$, set $x = y$ (so the cardinality of $V(G)$ is decreased by one).

\begin{lemma}\label{control3edgecut}
    If $G$ is a counterexample to Theorem \ref{mainthm} for which $3|E(G)| - 4|V((G)|$ is minimum, then $G$ does not have a set $Y \subseteq V$ with $|Y| \ge 2$ and $d(Y) = 3$ for which $G[Y]$ is balanced.  
\end{lemma}

\begin{proof}
    Suppose (for a contradiction) that such a set $Y$ exists and let $X = V(G) \setminus X$.  By Lemma \ref{getcubic} the graph $G$ is cubic so the edges in $E(X,Y)$ are pairwise nonadjacent.  It follows from this and our assumptions that we may modify the signature of $G$ so that every edge incident with a vertex in $Y$ is positive.  Choose an orientation of $G$ for which all edges in $E(X,Y)$ are directed from a vertex in $X$ to a vertex in $Y$.  Let $G_x$ $(G_y)$ be the oriented signed graph obtained from $G$ by identifying $X$ to a new vertex $x$ (identifying $Y$ to a new vertex $y$).  By our assumptions, $G_x$ has all {non-loop} edges positive {(i.e. it is an oriented graph, plus 0 or at least two negative loops incident to $x$)} and $G_y$ is flow-admissible (if $G_y$ were not flow-admissible, the same would be true of $G$).  Note that for any cubic graph $H$ we have $3|E(H)| - 4 |V(H)| = \frac{1}{2}|V(H)|$.  Using this property and the minimality of our counterexample, we may choose a nowhere-zero 8-flow, say $\phi_y$, of $G_y$.  By applying Lemma \ref{prespecify_vertex_integer} we may choose a nowhere-zero 8-flow, say $\phi_x$ of $G_x$ so that $\phi_x(e) = \phi_y(e)$ holds for every $e \in E(X,Y)$.  Combining $\phi_x$ and $\phi_y$ then gives a suitable nowhere-zero 8-flow of $G$.  This contradiction completes the proof.
    \end{proof}

\section{Finding a usable cycle}\label{3ec usable cycle}

In the following definition, we describe the central class of graphs of interest for us in establishing the proof of our main theorem. A \emph{subcubic} (signed) graph has maximum degree three.  

\begin{definition}
    Let $G$ be a signed graph.  We say that $G$ is \emph{well-behaved} if it is subcubic and satisfies the following two properties for every non-empty $X \subseteq V(G)$
\begin{enumerate}[(i)]
    \item $3 \le d(X) + \sum_{x \in X} (3 - \mathrm{deg}(x) )$
    \item If $G[X]$ is balanced and $|X| \ge 2$ then $4 \le d(X) + \sum_{x \in X} (3 - \mathrm{deg}(x) )$
\end{enumerate}
\end{definition}\index{well-behaved}

Note that by taking $X = V(G)$ in part (i) of the above definition, every well-behaved graph must have some vertices with degree less than 3. Another thing to note is that every induced subgraph of a well-behaved graph is also well-behaved. The next observation illuminates the reason for the above definition.  It follows immediately from our definitions together with Lemmas \ref{getcubic} and \ref{control3edgecut}.  

\begin{observation}\label{obs:well-behaved}
    If $G$ is a counterexample to Theorem \ref{mainthm} for which $3|E(G)| - 4|V((G)|$ is minimum, then every proper induced subgraph of $G$ is well-behaved.  
\end{observation}

In the next lemma, we require Corollary \ref{thm: fromTied}, which is a straightforward consequence of the following theorem found in \cite{devos2023cycles}.

If $G$ is a signed graph and $e_1, e_2 \in E(G)$, then we say that $e_1$ and $e_2$ are \emph{untied} if there exists a positive cycle $C^+$ and a negative cycle $C^-$ so that $\{e_1, e_2\} \subseteq E(C^+) \cap E(C^-)$. Otherwise we say that $e_1$ and $e_2$ are \emph{tied}.

\begin{theorem}[\cite{devos2023cycles}]
\label{main}
Let $G$ be a 3-connected signed graph and let $e_1,e_2 \in E(G)$ be distinct and not in parallel with any other edges.  Then $e_1$ and $e_2$ are tied in $G$ if and only if one of the following holds:
\begin{enumerate}
\item There exists a parallel class $F$ containing edges of both signs so that $F^+ = F \cup \{e_1, e_2\}$ is an edge-cut and $G - F^+$ is balanced,
\item $e_1, e_2$ are incident with a common vertex $v$ and $G-v$ is balanced,
\item $G - \{e_1, e_2\}$ is balanced.
\end{enumerate}
\end{theorem}

\begin{corollary}
\label{thm: fromTied}
    Let $G$ be a 3-connected, cubic signed graph and let $e_1,e_2 \in E(G)$ be distinct. If $G-\{e_1,e_2\}$ is unbalanced, then there exists a positive cycle containing both $e_1$ and $e_2$.
\end{corollary}

We say a generalized cycle (i.e.: a cycle or a single vertex) $C$ in a subcubic signed graph is \emph{good}\index{good cycle}\index{cycle!good} if $C$ is either a vertex of degree 0 or 1, or $C$ is a positive cycle containing at least two vertices of degree 2.  A generalized cycle is \emph{usable}\index{usable cycle}\index{cycle!usable} if it is either good, or is a negative cycle containing at most one vertex not of degree 2. We use $V_i(G)$ to denote the set of vertices of degree $i$ in $G$. A \emph{theta} graph is a 2-connected subcubic graph containing exactly two vertices of degree three.

\begin{lemma}\label{usable_cycle}
    If $G$ is a well-behaved, 2-connected signed graph, then either $G$ is a negative cycle, or $G$ contains a good cycle.
\end{lemma}
\begin{proof}
    Suppose for contradiction that the lemma does not hold.  We proceed in two cases.

    \begin{center}
        \includegraphics[]{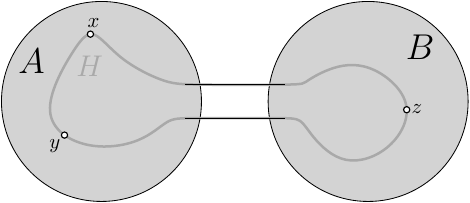}
    \end{center}

    In the first case, suppose $G$ is not a subdivison of a $3$-connected signed graph. If $G$ were itself a cycle then it would be usable, a contradiction, which means there is a $2$-edge-cut $E(A,B)$ that separates two degree-three vertices in $G$. Because $G$ is well-behaved, without loss of generality there exist $\{x,y\} \subseteq V_2(G) \cap A$, and $z \in V_2(G) \cap B$.  We may further suppose that $E(A,B)$ was chosen so that $B$ is minimal. By $2$-connectivity, $G$ has a subgraph $H$ that is either a cycle or a theta with $x,y,z \in V(H)$; and such that if $H$ is a theta, then each of $x,y,z$ are on different branches of $H$. If $H$ is a positive cycle or a theta, then $H$ contains a good cycle, a contradiction. Hence it must be that $H$ is a negative cycle. If $G[B]$ is unbalanced, then we may reroute the path of $H$ through $B$ (possibly missing $z$) so it is a positive cycle. This rerouted cycle contains $x$ and $y$ and so it is good, again a contradiction. Hence it must be that $G[B]$ is balanced. By property (ii) in the definition of well-behaved, there must exist $w \in V_2(G) \cap B$ with $w \neq z$.  If there is a cycle containing $w$ and $z$ in $G[B]$, then it is positive and good, a contradiction. Hence it must be that $G[B]$ has a cut-edge separating $w$ and $z$. But this cut-edge contradicts that $B$ was chosen to be minimal. This completes the first case.

    In the second case, suppose $G$ is a subdivision of a $3$-connected graph, say $H$. Let $S \subseteq E(H)$ be the set of edges that are subdivided to form $G$, and notice that by our connectivity assumptions $|S| \ge 2$. Let $e_1, e_2 \in S$. Since $G$ does not have a good cycle, by Corollary \ref{thm: fromTied} this means $G-\{e_1,e_2$\} is balanced. And so, again because $G$ does not have a good cycle, we may assume the signature of $H$ is such that $e_1$ is negative while all other edges are positive.
    If $|S|\ge 3$ then by $3$-connectivity of $H$, $G$ has a good cycle, a contradiction. Hence $S = \{e_1,e_2\}$. Similarly, if $e_2$ is subdivided at least two times then $G$ has a good cycle, and so it follows that $e_2$ is subdivided exactly once while $e_1$ is subdivided at least twice.  But now the set $X$ consisting of all vertices in $G$ except for the degree two vertices on the subdivided edge $e_1$ contradicts part (ii) of the definition of well-behaved.  This completes the second case and the proof.
\end{proof}

\begin{center}
        \includegraphics[]{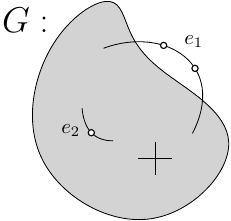}
    \end{center}

\section{Using an unbalanced theta}\label{3ecunbaltheta}

The main goal of this section is to prove Lemma \ref{theBeast}. That lemma will be used to satisfy a parity requirement on our set of cycles in Section \ref{csa}. (In that section, as in \cite{devos2023nowherezero}, we will need the number of even-length negative (not special) cycles in our list to be even). A main difference between this section and the parallel section in \cite{devos2023nowherezero} is that in \cite{devos2023nowherezero} we always found a pair of cycles satisfying our requirements, whereas a family of exception signed graphs (see Definition \ref{fish}) is allowed here. We will also call upon a special case of a theorem by Watkins and Mesner \cite{Watkins_Mesner_1967}, while the parallel section in \cite{devos2023nowherezero} is self-contained.

We say that a pair $(D,Q)$ of subgraphs of $G$ form a \emph{good theta-pair}\index{good theta-pair} if $D$ is a negative cycle, $Q$ is a path with both ends in $V(D)$ but internally disjoint from $V(D)$ and $|V(Q) \cap V_2(G)| \ge 2$.

\begin{lemma}\label{withGoodThetaPair}
    If $G$ is a well-behaved signed graph with a good theta-pair, then there exist a pair of cycles $C_0, C_1$ satisfying the following:
    \begin{enumerate}
        \item $C_1$ is negative and $C_0$ is good
        \item $|V_2(G) \cap V(C_1)| \ge 2$
        \item $G - V(C_0)$ and $G - V(C_1)$ have the same negative cycles.
    \end{enumerate}
\end{lemma}

\begin{proof}
    Choose a good theta-pair $(D,Q)$ for which $Q$ is maximal. Let $x,y$ be the ends of $Q$ and let $G' = G \setminus (V(Q) - \{x,y\} )$.  Let $H$ be the block of $G'$ containing the cycle $D$.
    
    \begin{figure}[h]
        \centering
        \includegraphics[width=0.25\linewidth]{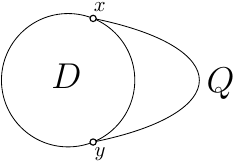}
    \end{figure}
    
      If there exists a negative cycle in $H$ not containing $x$, this cycle may be extended within $H$ to give a good theta-pair with a longer path---a contradiction. Hence every negative cycle in $H$ contains $x$. Let $C_0,C_1$ be the two cycles in the theta graph $D \cup Q$ containing $Q$ so that $C_0$ is positive and $C_1$ is negative.  It follows immediately from our construction that $C_0$ and $C_1$ satisfy all three required properties above.
\end{proof}

\begin{lemma}\label{cyclesContained}
    Let $G$ be an unbalanced, 2-connected, well-behaved signed graph with no good theta-pair and let $C_0$ be a good cycle.  Then there exists a negative cycle $C_1$ with $|V(C_1) \cap V_2(G)| \ge 1$ so that every negative cycle of $G - V(C_1)$ is contained in $G - V(C_0)$.
\end{lemma}

\begin{proof}
    Let $a,b \in V(C_0) \cap V_2(G)$ be distinct. Because $G$ is 2-connected and unbalanced, there must exist a path $P \subseteq G$ that has both ends in $V(C_0)$, is internally disjoint from $V(C_0)$, and so that $H = P \cup C_0$ is an unbalanced theta. Let $x,y \subseteq V(C_0)$ be the ends of $P$. Because $G$ does not have a good theta-pair it must be that $x,y$ interleave $a,b$ on $C_0$, and $V(C_0) \cap V_2(G) = \{a,b\}$. Let $Q \subseteq C_0$ be the $x,y$-path in $C_0$ which contains $a$, and suppose we have chosen $P$ so that $Q$ is minimal. Let $C_1$ be the negative cycle in $C_0 \cup P$ containing $b$. 

    \begin{figure}[h]
        \centering
        \includegraphics[width=0.35\linewidth]{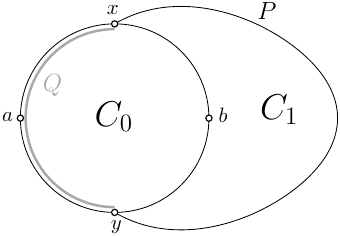}
    \end{figure}
        
    The vertex $b$ witnesses that $|V(C_1) \cap V_2(G)| \ge 1$.  To complete the proof, suppose (for contradiction) that $N$ is a negative cycle in $G - V(C_1)$ not contained in $G - V(C_0)$.  In this case, $N$ must intersect $Q - \{x,y\}$.  This means there is a path $P' \subseteq G$ with ends in $Q - \{x,y\}$ so that $P' \cup C_0$ makes an unbalanced theta. But this contradicts our choice of $P$.
      \end{proof}

We call a signed subcubic graph $G$ \emph{fragile} if $G$ contains an unbalanced theta, but for every good cycle $C$ the graph $G - V(C)$ has no unbalanced theta.
\begin{lemma}\label{fragileNegCyc}
    Let $G$ be a 2-connected, fragile, well-behaved signed graph with no good theta-pair, and let $C \subseteq G$ be a good cycle. Then every negative cycle $N$ in $G - V(C)$ has $|V_3(G) \cap V(N)| = 2$ and $|V_2(G) \cap V(N)| \ge 1$. 
\end{lemma}

\begin{proof}
Let $H$ be a component of $G' = G - E(C)$ with a negative cycle $N$. We will consider the block structure of $H$. Note that because $G$ is fragile, $N$ is itself a block of $H$. Because $G$ is subcubic and $2$-connected, every leaf-block of $H$ is a single edge with a vertex in $V(C)$. This means that if $|V(N) \cap V_3(G)| \ge 3$, then there would be three vertex-disjoint paths from $V(N)$ to $V(C)$ as in Figure \ref{fig:frag1}. But then using two of those paths, there would be a good theta-pair containing $N$ and a path $Q$ with two degree-2 vertices on $C$, a contradiction. Thus $|V(N) \cap V_3(G)| \le 2$. Since at least two degree-3 vertices are required by 2-connectivity, it must be that $|V(N) \cap V_3(G)| =2$. But then because $G$ is well-behaved it follows that $|V(N) \cap V_2(G)| \ge 1$.

    \begin{figure}[h]
        \centering
        \includegraphics[width=0.35\linewidth]{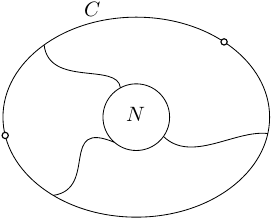}
        \label{fig:frag1}
    \end{figure}
\end{proof}

\begin{lemma}\label{atMost2negCyc}
    If $G$ is a $2$-connected, fragile, well-behaved signed graph with no good theta-pair, then for every good cycle $C$ the graph $G - V(C)$ has at most two negative cycles.
\end{lemma}

\begin{proof}
     Let $G' = G - E(C)$. By Lemma \ref{fragileNegCyc}, every negative cycle $D$ in $G'$ has $|V_3(G) \cap V(D)| = 2$ and $|V_2(G) \cap V(D)| \ge 1$.

    \begin{figure}[H]
        \centering
        \includegraphics[width=0.35\linewidth]{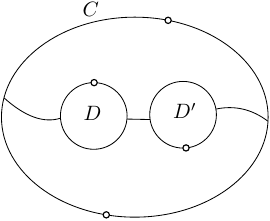}
        \label{fig:f1}
    \end{figure}

    If a component $H$ of $G'$ contained two negative cycles $D,D'$, then there would be a good theta-pair using $D$ and a path $Q$ containing a degree-2 vertex from $D'$ and a degree 2 vertex from $C$, a contradiction. Hence every component of $G'$ has at most one negative cycle.

    Now, suppose for contradiction that $G'$ has three negative cycles, say $D_1, D_2, D_3$. By 2-connectivity of $G$, for $1 \le i \le 3$ choose a pair of vertex disjoint paths from $V(D_i)$ to $V(C)$ and let $H_i$ be the subgraph of $G$ made by the union of $D_i$ and the two paths associated with it. Observe that by the previous paragraph, $H_1, H_2, H_3$ are pairwise vertex disjoint. For $1 \le i \le 3$ let $\{x_i,y_i\} = V(C) \cap V(H_i)$. Note that $H_i$ contains paths between $x_i$ and $y_i$ of both signs, and that at least one of those two paths contains a vertex in $V_2(G)$.

    \begin{figure}[H]
        \centering
        \includegraphics[width=0.7\linewidth]{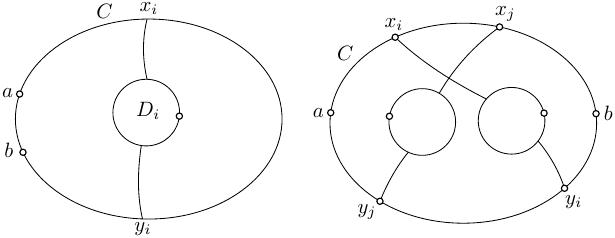}
        \label{fig:f34}
    \end{figure}

    Now, because $C$ is good, there exist $a,b \in V_2(G) \cap C$. It must be that, for each $1 \le i \le 3$, the vertices $x_i, y_i$ interleave $a,b$ on $C$, because otherwise there would be a good theta-pair using $D_i$ and a path $Q \subseteq H_i \cup C$ which contains $a,b$ as in the left side of the above figure. It also must be that $x_i,y_i$ do not interleave $x_j,y_j$ for $i \neq j$, because otherwise there would be a good theta-pair using $D_i$ and a path $Q\subseteq H_i \cup H_j \cup C$ containing a vertex in $V_2(G) \cap V(D_j)$, and a vertex in $V_2(G) \cap V(C)$, as in the right side of the above figure.

    \begin{figure}[H]
        \centering
        \includegraphics[width=0.38\linewidth]{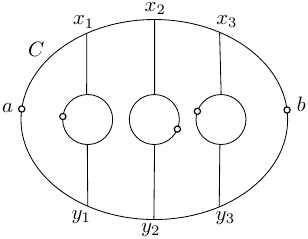}
        \label{fig:fragile5}
    \end{figure}

    Hence it follows (wlog) that the vertices $a,x_1,x_2,x_3,b,y_3,y_2,y_1$ occur in that cyclic order around $C$. But now $H_1 \cup C$ has an unbalanced theta while $H_2 \cup H_3 \cup C$ has a theta with a degree-2 vertex on each branch ($x_3$ and $y_3$ are the vertices of degree 3 in that theta), and so that these subgraphs are disjoint. Since the later subgraph contains a good cycle, this contradicts that $G$ is fragile and completes the proof.
\end{proof}

For the next lemma, we define the family of exception signed graphs mentioned at the start of this section.

\begin{definition}\label{fish}
    We say a signed graph $G$ is a \emph{fish} if $G$ is formed by taking the union of a theta $H$ whose branches have length two, two, and three; and a path $P$ of odd-length at least three which is internally disjoint from $H$ but whose ends are the adjacent degree-2 vertices of $H$. And the signature of $G$ is such that $H$ is balanced but $H \cup P$ is unbalanced. See Figure \ref{fig:fish}.
\end{definition}

\begin{figure}[H]
    \centering
    \includegraphics{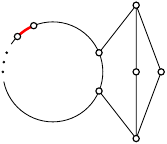}
    \caption{A family of \emph{fish} signed graphs: signed graphs having an equivalent signature with a single negative edge shown in red. The three black dots represent an even number (zero or more) of consecutive degree-two vertices.}
    \label{fig:fish}
\end{figure}

We also require the following theorem of Watkins and Mesner \cite{Watkins_Mesner_1967}. 

\begin{theorem}[\cite{Watkins_Mesner_1967}]\label{thm:mw}
    Let $G$ be a $2$-connected subcubic graph and let $x_1,x_2,x_3 \in V(G)$ be distinct. Either there exists a cycle $C \subseteq G$ with $x_1,x_2,x_3 \in V(C)$ or there is a partition of $V(G)$ into $\{Y_1,Y_2,X_1,X_2,X_3\}$ with $x_i \in X_i$ for $1 \le i \le 3$ satisfying 
    \begin{itemize}
        \item $e(X_i,Y_j) = 1$ for $1 \le i \le 3$ and $1 \le j \le 2$, and
        \item $e(Y_1,Y_2) = 0 = E(X_i,X_j)$ for $1 \le i < j \le 3$.
    \end{itemize}
\end{theorem}

\begin{figure}[H]\label{fig:chooseParity}
    \centering
    \includegraphics{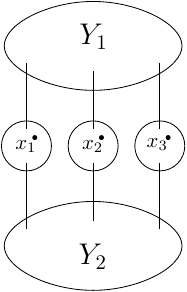}
    \label{fig:wm}
\end{figure}

\begin{lemma}\label{disjNegcycfromGood}
    Let $G$ be a 2-connected, fragile, well-behaved signed graph with no good theta-pair, and assume that there exists in $G$ a good cycle disjoint from some even-length negative cycle.  Then either $G$ is isomorphic to fish, or there exist cycles $C_0,C_1$ satisfying the following:
         \begin{enumerate}
        \item $C_1$ is negative and $C_0$ is good
        \item $|V_2(G) \cap V(C_1)| \ge 1$
        \item $G - V(C_0)$ and $G - V(C_1)$ have the same negative cycles.
        \end{enumerate}
\end{lemma}

\begin{proof}
  Choose a good cycle $C$ and an even-length negative cycle $N$ for which $V(C) \cap V(N) =~\emptyset$. By Lemma \ref{fragileNegCyc}, let $\{x,y\} = V_3(G) \cap V(N)$. Note, of course, that this means there is a 2-edge cut separating $V(N)$ and $V(G) \setminus V(N)$. Again by Lemma \ref{fragileNegCyc}, $|V_2(G) \cap V(N)| \ge 1$, but since $N$ has even-length it must be that $|V_2(G) \cap V(N)| \ge 2$. By 2-connectivity, we may choose two edge-disjoint paths $P_1, P_2$ from a vertex in $V_2(G) \setminus V(N)$ and $x,y$. 
  This means that if there was a degree-2 vertex on both of the $x,y$-paths in $N$, then $N \cup P_1 \cup P_2$ would be a theta which contains a pair of cycles satisfying the lemma.  Hence we may assume $x$ and $y$ are adjacent in $N$, and we let $Q \subseteq N$ be the longest $x,y$-path in $N$.  Note that $|V(Q) \cap V_2(G)| \ge 2$.

    \begin{figure}[h]
      \centering
      \includegraphics[width=0.3\linewidth]{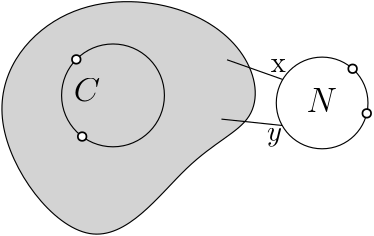}
      \label{fig:enter-label}
  \end{figure}
  
  If $G - E(Q)$ has a negative cycle $N'$, then by 2-connectivity there is a good theta-pair using $N'$ and a path containing $Q$, a contradiction. Hence $G - E(Q)$ is balanced.
  In particular, $G$ can be signed with a single negative edge contained in $E(Q)$.

    
    Now, modify $G$ to form $G'$ by identifying $N$ to a new (degree-2) vertex called $x_1$.  Let $x_2,x_3$ be distinct degree-two vertices in $V(C)$.  The graph $G'$ cannot have a cycle containing $x_1, x_2, x_3$ as this would give a good theta-pair (contradicting our assumption).  Therefore, by applying Theorem \ref{thm:mw} to the graph $G'$ we deduce that the original graph $G$ has a vertex partition $\{Y_1,Y_2,X_1,X_2,X_3\}$ with $N \subseteq G[X_1]$, $x_i \in X_i$ for $2 \le i \le 3$, and so that 
    $e(X_i,Y_j) = 1$ for $1 \le i \le 3$, $1 \le j \le 2$, and $e(Y_1,Y_2) = 0 = E(X_i,X_j)$ for $1 \le i < j \le 3$. Choose this structure so that $|V(Y_1) \cup V(Y_2)|$ is maximum. This means that each $X_i$ is 2-connected or a single vertex. 

    For $ i = 2,3$, it must be that $|V_2(G) \cap X_i| = 1$, because otherwise there would be a good cycle in $X_i$ contradicting that $G$ is fragile. Because $G$ is well-behaved, this means that $X_2 = \{x_2\}$, and $X_3 = \{x_3\}$. If $N \neq G[X_1]$, then $X_1$ contains an unbalanced theta which is disjoint from a good cycle using $x_2, x_3$, contradicting that $G$ is fragile. Hence $G[X_1] = N$.
    
    Similarly, for $i = 1,2$ if $Y_i \cap V_2(G) \neq \emptyset$, then $G$ has a good theta-pair using $N$ and a path containing a degree-2 vertex in $Y_i$ and one of the degree-2 vertices in $X_2$ or $X_3$. Because $G$ is well-behaved this means each $Y_i$ is also a single vertex. But now the graph $G$ is isomorphic to fish, which completes the proof.
    \end{proof}

With the above lemmas in place, we can now prove the following, which we will use going forward.

\begin{lemma}\label{theBeast}
    Let $G$ be a 2-connected, fragile, well-behaved signed graph. Then either $G$ is isomorphic to fish, or $G$ has a pair of cycles $C_0,C_1$ satisfying the following:
     \begin{enumerate}
        \item\label{g1} $C_1$ is negative and $C_0$ is good
        \item\label{g2} $G - V(C_1)$ has no unbalanced theta
        \item\label{g3} $G - V(C_0)$ and $G - V(C_1)$ have the same even length negative cycles.
        \item\label{g4} \begin{enumerate}
            \item $|V_2(G) \cap V(C_1)| \ge 2$ or
            \item $|V_2(G) \cap V(C_1)| \ge 1$ and the graph $G - V(C_1)$ has at most two negative cycles
            \end{enumerate}
    \end{enumerate}
\end{lemma}

\begin{proof}
We proceed in two cases. First, suppose $G$ has a good theta-pair. Then by Lemma \ref{withGoodThetaPair} $G$ has a negative cycle $C_1$ and a good cycle $C_0$ so that $|V_2(G) \cap V(C_1)| \ge 2$, and the graphs $G - V(C_0)$ and $G - V(C_1)$ have the same negative cycles. Clearly requirements \ref{g1}, \ref{g3}, and \ref{g4} are satisfied in this case. But an unbalanced theta is just a union of two negative cycles, and $G - V(C_0)$ and $G - V(C_1)$ have the same negative cycles. This means that because $G$ is fragile (and so $G - V(C_0)$ has no unbalanced theta) $G- V(C_1)$ also does not have an unbalanced theta, and requirement \ref{g2} is satisfied.

Now suppose $G$ has no good theta-pair. We proceed in two subcases.
In the first subcase, suppose there is a good cycle $C_0$ so that every negative cycle in $G - V(C_0)$ has odd-length. Then by Lemma \ref{cyclesContained} there is a negative cycle $C_1$ with $|V_2(G) \cap V(C_1)| \ge 1$ so that every negative cycle of $G - V(C_1)$ is contained in $G - V(C_0)$. Since by assumption of this subcase $G - V(C_0)$ has no even-length negative cycles, the same is true for $G - V(C_1)$, Hence $C_0$ and $C_1$ satisfy requirement \ref{g1}, \ref{g3}, and the first part of \ref{g4} (b) of the lemma. As before, because $G$ is fragile $G- V(C_0)$ has no unbalanced theta, which means the same is true for $G-V(C_1)$. Thus $C_0$ and $C_1$ also satisfy \ref{g2}. Finally, Lemma \ref{atMost2negCyc} means that $G - V(C_0)$ (and so also $G - V(C_1)$) has at most two negative cycles, meaning the final part of $\ref{g4}$ (b) is satisfied.

In the second subcase, by Lemma \ref{usable_cycle} let $C_0$ be a good cycle in $G$. Since we are not in the first subcase, let $N \subseteq G$ be an even-length negative cycle disjoint from $C_0$. If $G$ is isomorphic to fish, the lemma is proved. And so we may assume $G$ is not isomorphic to fish which means by Lemma \ref{disjNegcycfromGood} there exist a negative cycle $C_1$ and a good cycle $C_0$ in $G$ so that $|V_2(G) \cap V(C_1)| \ge 1$, and the graphs $G - V(C_0)$ and $G - V(C_1)$ have the same negative cycles. Recalling that $G$ is fragile, this means requirements \ref{g1}, \ref{g2}, \ref{g3}, and the first part of \ref{g4} (b) of the lemma are satisfied by $C_0$ and $C_1$. And, again by Lemma \ref{atMost2negCyc} the final part of $\ref{g4}$ (b) is satisfied, completing the proof.
\end{proof}

\section{Constructing a sequence of generalized cycles using the Cycle Selection Algorithm}\label{csa}

With the lemmas of the previous two sections in place, we  describe an algorithm to select our list of cycles. We will also require the following lemma from \cite{devos2023nowherezero}.

\begin{lemma}[\cite{devos2023nowherezero}]\label{lem:firstCycle}
Let $G$ be a 3-connected signed graph that has two vertex-disjoint negative cycles.  Then one of the following holds:
\begin{enumerate}
\item $G$ has vertex-disjoint negative cycles $C_1$, $C_2$ so that $V(C_1) \cup V(C_2) =V(G)$, or
\item $G$ has a negative cycle $C$ so that $G \setminus V(C)$ has an unbalanced theta.
\end{enumerate}
\end{lemma}

\vspace{1cm}

\hrule

\smallskip

\hrule

\medskip

    \noindent{\textbf{Cycle Selection Algorithm}}
    
    \medskip
    
    \noindent{\textbf{Input:} A signed 3-edge-connected cubic graph $G$ with two disjoint negative cycles}
    
    \noindent{\textbf{Output:} A non-empty ordered list $\mathcal{C} = C_1, C_2, \dots , C_t$ of `cycles', where each cycle in $\mathcal{C}$ is classified one of four ways: positive, negative ordinary, negative special, or fish}

    \medskip

\noindent{\textbf{Pre-process step: }One of the two cases in Lemma \ref{lem:firstCycle} holds for $G$. If it is the first case, then let $C_1,C_2$ be the two cycles given in that case. Declare both of them to be \underline{negative special} and \textbf{return $\mathcal{C} = C_1,C_2$}.}

The second case of Lemma \ref{lem:firstCycle} must hold. Let $C_1$ be the negative cycle supplied in that case, and declare it to be \underline{negative special}. The remainder of the algorithm is recursive. Begin with $\mathcal{C} = C_1$.

\medskip

\noindent{\textbf{Recursive step: }Given $\mathcal{C} = C_1, \dots, C_{i - 1}$. Let $G' = G - \cup_{C \in \mathcal{C}}V(C)$. Note that our assumptions imply that $G'$ is well-behaved.}

        \begin{enumerate}
            \item If $G'$ is empty \textbf{return} $\mathcal{C}$.
            \item Let $H$ be a leaf-block in the block-cut tree of $G'$, choose $H$ so that $G' - V(H)$ has an unbalanced theta if possible.

            (From this point forward, $H$ has a usable cycle by Lemma \ref{usable_cycle} and because a vertex of degree 0 or 1 is usable.)
            
            \item \label{step3} If $G'$ does not have an unbalanced theta, let $C_i$ be a usable cycle in $H$. If $C_i$ is negative, then declare that $C_i$ is \underline{negative ordinary}. Otherwise declare that $C_i$ is \underline{positive}. Then, \textbf{goto} Recursive step with $\mathcal{C} \cup \{C_i\}$.
            
            (From this point forward, $G'$ has an unbalanced theta.)
            \item \label{step5} If $H$ is a negative cycle, then it is usable. Declare that $C_i= H$ is \underline{negative ordinary}, and \textbf{goto} Recursive step with $\mathcal{C} \cup \{C_i\}$.
            
            (From this point forward $H$ has a good cycle by Lemma \ref{usable_cycle}.)
            
            \item \label{step7}Let $C_i$ be a good cycle in $H$, choosing so that $G'- V(C_i)$ has an unbalanced theta if possible.
            
            \item\label{step6} If $G' - V(C_i)$ has an unbalanced theta, declare that $C_i$ is \underline{positive} and \textbf{goto} Recursive step with $\mathcal{C} \cup \{C_i\}$.

            (From this point forward, $G'$ is fragile by definition, and $G' = H$ is 2-connected.)

            \item Let $k$ be the number of even-length negative ordinary cycles in $\mathcal{C}$.
            
            \item \label{unbalThetaStep}
            \begin{enumerate}
                \item \label{fishStep}If $G'$ is switching equivalent to a fish graph proceed as follows: If $k$ is odd, declare that $C_i = G'$ is \underline{fish} and \textbf{goto} Recursive step with $\mathcal{C} \cup \{C_i\}$. If $k$ is even, let $C_i$ be the good cycle in $G'$, declare that $C_i$ is positive and \textbf{goto} Recursive step with $\mathcal{C} \cup \{C_i\}$.

                \item \label{endStep}$G'$ is not switching equivalent to a fish graph.
                \begin{itemize}
                    \item Let $C_0, C_1$ be the cycles supplied by Lemma \ref{theBeast}.
                    \item Let $j$ be the number of even-length negative cycles in $G' - V(C_0)$ (which is equal to the number of even-length negative cycles in $G' - V(C_1)$).
                \end{itemize}

                If $k + j$ is even, let $C_i$ be $C_1$ and declare $C_i$ to be \underline{negative special}. Otherwise let $C_i$ be $C_0$ and declare $C_i$ to be \underline{positive}. Then, \textbf{goto} Recursive step with $\mathcal{C} \cup \{C_i\}$.
            \end{enumerate}

        \end{enumerate} 

\hrule

\smallskip

\hrule

\vspace{1cm}

The following lemmas describe properties of the Cycle Selection Algorithm.

\begin{lemma}\label{lem:cycSelAlg}
    The Cycle Selection Algorithm terminates. The members of the returned list $\mathcal{C}$ are vertex-disjoint and $V(G) = \cup_{C \in \mathcal{C}}V(C)$.
\end{lemma}
\begin{proof}
The first sentence is true because an element from a finite set is added to $\mathcal{C}$ at each recursive step. The second sentence follows from Lemma \ref{lem:firstCycle} for the pre-process step, and for the recursion it follows by definition of $G'$, and because $G'$ is empty when the algorithm terminates.    
\end{proof}

\begin{lemma}\label{lem:step8}
    The Cycle Selection Algorithm either exits in the pre-process step, or reaches Step \ref{unbalThetaStep} exactly once.
\end{lemma}

\begin{proof}
    If the algorithm doesn't exit in the pre-process step, then it begins the recursion with an unbalanced theta in the graph $G'$. This unbalanced theta remains in the graph $G'$ until step \ref{unbalThetaStep} is reached, and so step \ref{unbalThetaStep} is reached at least once. At step \ref{unbalThetaStep}, one of steps \ref{fishStep} or \ref{endStep} is executed and adds a cycle to $\mathcal{C}$ which removes all unbalanced thetas from the next graph to be processed. (This is straightforward to see in step \ref{fishStep}, and in step \ref{endStep} it follows from Lemma \ref{theBeast}). Since $G'$ must have an unbalanced theta to reach step \ref{unbalThetaStep}, it follows that step \ref{unbalThetaStep} can only be reached once.
\end{proof}
    
\begin{lemma}\label{lem:special}
    In the Cycle Selection Algorithm, the first cycle $C_1$ in the returned list $\mathcal{C}$ is negative special, and at most one other cycle in $\mathcal{C}$ is negative special. 
\end{lemma}
\begin{proof}
    The first part is straightforward. By Lemma \ref{lem:step8}, step \ref{unbalThetaStep} is encountered at most once in the algorithm. In particular this means step \ref{endStep} is executed at most once. Since step \ref{endStep} is the only recursive step to add a negative special cycle to $\mathcal{C}$, the lemma follows.
\end{proof}

\begin{lemma}\label{lem:cycSelnegspec}
    If $\mathcal{C} = C_1, \dots, C_t$ is returned by the Cycle Selection Algorithm, and $C_j\subseteq \mathcal{C}$ is type negative special with $j \neq 1$, then either $e(\cup_{i = 1} ^{j-1} V(C_i), V(C_j)) \ge 2$, or $e(\cup_{i = 1} ^{j-1} V(C_i), V(C_j)) \ge 1$ and at most two negative cycles come after $C_j$ in the list $\mathcal{C}$.
\end{lemma}

\begin{proof}
    If the algorithm exits in the pre-process step, then $t=2$ and the lemma holds. And so we may assume $C_j$ is added to $\mathcal{C}$ in step \ref{endStep} because that is the only recursive step to add a negative special cycle. But in step \ref{endStep}, the negative special cycle is supplied by Lemma \ref{theBeast}. This lemma follows from Part \ref{g4} of Lemma \ref{theBeast} because the input graph $G$ is cubic.
\end{proof}

\begin{lemma}\label{lem:cycSelFish}
    If $\mathcal{C} = C_1, \dots, C_t$ is returned by the Cycle Selection Algorithm, and $C_j\subseteq \mathcal{C}$ is type fish, then $j = t$.
\end{lemma}
\begin{proof}
    Step \ref{fishStep} of the algorithm is the only step to add a fish-type cycle to $\mathcal{C}$. At that point, the graph $G'$ being processed is a fish graph. And so $C_j$ is the last cycle to be added to~$\mathcal{C}$.
\end{proof}

\begin{lemma}\label{lem:ordinary}
    In the Cycle Selection Algorithm, every negative ordinary cycle $C_j$ in the returned list $\mathcal{C} = C_1, C_2, \dots, C_t$ is induced and satisfies $e(V(C_j), \cup_{i = j+1}^{t} V(C_i) ) \le 1$. (In other words, all but at most one of the edges leaving a negative ordinary cycle go to the cycles earlier in the list).
\end{lemma}
\begin{proof}
    The algorithm adds a negative ordinary cycle $C_j$ to $\mathcal{C}$ only in steps \ref{step3} and \ref{step5}. In both cases, $C_j$ is a negative usable cycle and so has at most one vertex not in $V_2(G')$ by definition. Because the input graph is cubic the lemma follows.
\end{proof}

\begin{lemma}\label{lem:positive}
    In the Cycle Selection Algorithm, every positive cycle $C_j$ in the returned list $\mathcal{C} = C_1, C_2, \dots, C_t$ has two distinct vertices $x,y \in V(C_j)$ so that $x$ has a neighbour in $\cup_{i = 1}^{j-1} V(C_i)$ and one of the following is true.
    \begin{enumerate}
        \item $y$ also has a neighbour in $\cup_{i = 1}^{j-1} V(C_i)$.
        \item $y$ is incident with a cut edge in $\cup_{i = j}^{t} V(C_i)$. 
    \end{enumerate}
\end{lemma}

\begin{proof}
    The algorithm adds a positive cycle $C_j$ to $\mathcal{C}$ only in steps \ref{step3}, \ref{step6} and \ref{endStep}. In all cases the cycle added is good. But a good cycle is either a single vertex of degree at most one, or has at least two vertices of degree two, and so the lemma follows because the input graph is cubic.
\end{proof}

    

\begin{lemma}\label{CYcSelLast}
    If $\mathcal{C}$ is the list returned by the Cycle Selection Algorithm, the number of even-length negative ordinary cycles plus the number of special cycles in $\mathcal{C}$ is even.
\end{lemma}
\begin{proof}
    Let $\mathcal{C} = C_1, C_2, \dots, C_t$ be the list returned by the algorithm on input graph $G$.
    If the algorithm finishes in the pre-process step, then the lemma clearly holds. And so we may assume that the recursive step is reached. By Lemma \ref{lem:step8}, step \ref{unbalThetaStep} is reached exactly once, and therefore executes exactly one of \ref{fishStep} or \ref{endStep}. If \ref{fishStep} is executed, then either a fish is added or both a positive cycle and an even-length negative ordinary cycle is added to $\mathcal{C}$ (the negative cycle is added on the next step in the recursion). Either way this fixes the number of even-length negative ordinary cycles in $\mathcal{C}$ to be odd. Since there is exactly one negative special cycle in $\mathcal{C}$ the lemma follows.
    
    If \ref{endStep} is executed, note that in the following recursive call $G'$ is a graph with no unbalanced theta, and so every negative cycle is its own block. It follows that every one of those negative cycles will be added to $\mathcal{C}$. The choice of adding either a negative special cycle or a positive cycle in that step fixes the parity required for this lemma.
\end{proof}

The following lemma will be used to establish sufficient connectivity to find a perfect matching in a certain auxiliary graph later. Let $G = (V,E)$ be a connected graph and let $X \subseteq V$. Denote by $\delta(X)$ the set of edges with one end in $X$ and the other in $V\setminus X$. As in \cite{devos2023nowherezero}, we say that a cycle $C$ \emph{straddles} the edge-cut $\delta(X)$ if $C$ contains an edge with both ends in $X$ and an edge with both ends in $V \setminus X$.

\begin{lemma}\label{lem:negCycStraddle}
    Let $G$ be a signed 3-edge-connected cubic graph with two disjoint negative cycles, and let $\mathcal{C} = C_1, \dots, C_t$ be the list of cycles returned when $G$ is input to the Cycle Selection Algorithm. Let $\mathcal{N}$ be the set of all but the last two even-length negative ordinary cycles in $\mathcal{C}$. Then for every $C_i,C_j \in \mathcal{N}$ and every edge cut $S \subseteq E(G)$
\begin{enumerate}
    \item if $C_i$ straddles $S$, then $|S| \neq 3$, and
    \item if $C_i$ and $C_j$ straddle $S$, and $i \neq j$, then $|S| \neq 4$.
\end{enumerate}
\end{lemma}

\begin{proof}
  Suppose for contradiction part two is not true.  Then there exists $S = E(A,B)$ with $|S| = 4$ so that both $C_i$ and $C_j$ straddle $S$.  By Lemma \ref{lem:special} $C_1 \neq C_i,C_j$ because $C_1$ is negative special. This means $E(C_1) \cap S = \emptyset$ and so without loss let $V(C_1) \subseteq A$. Let $k$ be the minimum index for which $V(C_k) \cap B \neq \emptyset$ and note that $k \le \min\{i,j\}$.  At the point when $C_k$ is selected, we are operating in the graph $G - \cup_{h=1}^{k-1} V(C_h)$ and in this graph all vertices in $B$ have degree 3.  Every cycle selected by the algorithm has at least one degree 2 vertex (by Lemma \ref{lem:positive} a positive cycle has at least one degree 2 vertex, by Lemma \ref{lem:ordinary} a negative ordinary cycle has at least two degree 2 vertices, and by Lemma \ref{lem:cycSelnegspec} a negative special cycle has at least one degree 2 vertex) so it must be that $k = \min\{i,j\}$.  But then $C_k$ would be an ordinary negative cycle containing at least two degree 3 vertices, which is a contradiction to Lemma \ref{lem:ordinary}. This completes the proof of part two.
  
Next suppose (for a contradiction) that part one is false and let $S = E(A,B)$ be a 3-edge-cut straddled by the negative ordinary cycle $C_i$.  Define $S = \{e,f,f'\}$ where $S \cap E(C_i) = \{f,f'\}$ and note that because $G$ is 3-edge-connected and cubic, the edges $e,f,f'$ are pairwise non-adjacent.  
As before, we may assume $V(C_1) \subseteq A$ and we let $k$ be the minimum index for which $V(C_k) \cap B \neq \emptyset$ (again note that $k \le i$).  Let $v$ be the end of $e$ in $B$.  Note that at the point when $C_k$ is selected, we are operating in the graph $G - \cup_{h=1}^{k-1} V(C_h)$ and in this graph all vertices in $B$ have degree 3 except possibly $v$.  

First consider the case that $k = i$ and let $u,u'$ be the ends of the edges $f,f'$ in $B$.  When $C_k$ is selected both $u$ and $u'$ are degree 3, but since $C_k$ is negative ordinary this is a contradiction to Lemma \ref{lem:ordinary}. So we must have $k < i$ and hence $V(C_k) \subseteq B$. This means that $C_k$ has at most one vertex of degree 2 when it is selected.

Observe that there are only two types of cycles that can be selected by our process with at most one vertex of degree two: negative special and positive. If $C_k$ were negative special, it would have at most two negative cycles that follow it by Lemma \ref{lem:cycSelnegspec}. But this contradicts that $k \le i$ and $C_i \in \mathcal{N}$ by definition of $\mathcal{N}$. And so it must be that $C_k$ is positive. Because at most one vertex in $V(C_k)$ has a neighbour in $A$, by Lemma \ref{lem:positive} it must be that there is a vertex $y \in V(C_k)$ incident to a cut-edge in $G[B]$. But this contradicts that $G$ is 3-edge-connected. And so $C_k$ cannot be positive, a contradiction. This completes the proof of part one, and the lemma.     
\end{proof}

\section{Forming a ${\mathbb Z}_3$-preflow}\label{3ecZ3pre}

As in the parallel section in \cite{devos2023nowherezero}, we will begin our construction of a nowhere-zero 8-flow with a $\Z_3$-preflow. For a subcubic bidirected graph $G = (V,E)$ and abelian group $A$, a function $\phi : E \rightarrow A$ is a \emph{preflow}\index{preflow} when $\partial \phi (v) \neq 0$ if and only if $\mathrm{deg}(v) \in \{1,2\}$. We say $\partial\phi(v)$ is the \emph{boundary} of $\phi$ at $v$.

We will require the following lemma from \cite{devos2023nowherezero}.

\begin{lemma}[\cite{devos2023nowherezero}]
\label{cycle_boundary1}
Let $C$ be a cycle, and let $b: V(C) \rightarrow {\mathbb Z}_3$.  If $C$ is negative, or all edges of $C$ are positive and $\sum_{v \in V(C)} b(v) = 0$,
then there exists $\tau : E(C) \rightarrow {\mathbb Z}_3$ so that $\partial \tau = b$.
\end{lemma}

The following lemma is similar to the parallel lemma in \cite{devos2023nowherezero}; the main differences are that here there is a fish-type `cycle', and positive cycles are no longer guaranteed to have two edges to cycles before them in the list $\mathcal{C}$. 

We say the unique edge in fish that is not in a 2-edge cut is \emph{distinguished} (as shown in Figure \ref{fig:fishFlow}).

\begin{lemma}\label{lem:buildPreflow}
        Let $G$ be a signed 3-edge-connected cubic graph with two disjoint negative cycles, and let $\mathcal{C} = C_1, \dots, C_t$ be the list of cycles returned when $G$ is input to the Cycle Selection Algorithm.  Let $G^*$ be an orientation of the graph obtained from $G$ by subdividing one edge on every negative special and even-length negative ordinary cycle in $\mathcal{C}$.  Then there exists a preflow $\phi: E(G^*) \rightarrow \mathbb{Z}_3$ so that $\phi(e) = 0$ only when $e$ lies in (a subdivision of) a negative special cycle, a positive cycle, or $e$ is the distinguished edge in fish. 
\end{lemma}

\begin{proof}
    For every $C_i \in \mathcal{C}$ that is a cycle in $G$, let $C_i^*$ be the corresponding cycle in $G^*$ which is obtained from $C_i$ by at most one subdivision. Denote the corresponding list of these cycles by $\mathcal{C}^*$. We will use the same classification: positive, negative ordinary, negative special, or fish for both $C_i$ and $C_i^*$. 

    For $1 \le k \le t$ let $U_k =  \bigcup_{i=k+1}^t V(C_i^*)$ and $D_k = \bigcup_{i=1}^{k-1}V(C_i^*)$ ($U$ for up and $D$ for down). Let $F_k$ denote $E(V(C_k^*), D_k)$ and $E_k$ denote $E(V(C_k^*), U_k)$.

    Before we start, we choose a signature of $G^*$ obeying certain qualities on the cycles $C_k^* \in \mathcal{C^*}$. Note that by switching only on the vertices of $C_k^*$, we may change the sign of the edges in $E(C_k^*)$ without affecting the sign of the edges of any other cycle in $\mathcal{C^*}$. First, if $C_k^*$ is positive, we choose a signature so that every edge of $E(C_k^*)$ is positive. Similarly, if $C_k^*$ is negative ordinary, because our subdivisions have ensured that $C_k^*$ has odd length, then we choose a signature so that every edge in $E(C_k^*)$ is negative. We make no restriction on the edges of $C_k^*$ if it is negative special, but if $C_k^*$ is fish then we choose a signature with exactly one negative edge (as in Figure \ref{fig:fish}).

    We will construct a preflow $\phi: E(G^*) \to \Z_3$. Begin with $\phi$ having empty domain, and we will extend the domain of $\phi$ in steps, starting with the edges incident to a vertex in $V(C_t^*)$, then $V(C_{t-1}^*)$ and proceeding backwards through the list $\mathcal{C^*}$ finishing with the edges incident to a vertex in $V(C_1^*)$ at which point the domain of $\phi$ will be $E(G^*)$. For $i = t \dots 1$, at each step $i$ we will ensure that 
    \begin{itemize}
        \item every $v \in V(C_i^*) \cup U_i$ satisfies the boundary requirements of a preflow (that is $\partial\phi(v) =~0$ if and only if the degree of $v$ is three), and
        \item for each edge $e$ with both ends in $[V(C_i^*) \cup U_i]$, $\phi(e) = 0$ only when $e$ lies in a negative special cycle, a positive cycle, or is the distinguished edge in fish.
    \end{itemize}
    At each step $i$ we proceed according to the type of $C_i^* \in G^*$.

    \medskip

    \noindent{Case 1: $C_k^*$ is fish.}

    \medskip

    By Lemma \ref{lem:cycSelFish}, $k=t$ and $C_k^*$ is the first graph to be processed. Without loss of generality, let $C_k^*$ be partially oriented as in Figure \ref{fig:fishFlow} (the orientation of the distinguished edge is immaterial). For an edge $e \in C_k^*$, set $\phi(e) = 0$ if $e$ is the distinguished edge, and $\phi(e) = 1$ otherwise. Because $k=t$, it follows that $E_k = \emptyset$. Finally for each $f \in F_k$ we may choose $\phi(f) = \pm 1$ so its end in $V(C_k^*)$ has boundary zero. This satisfies the requirements of a preflow because every vertex $v \in C_k^*$ has boundary $\partial\phi= 0$.

    \begin{figure}[h]
        \centering
        \includegraphics[scale = 1.5]{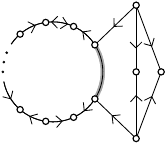}
        \caption{A fish graph with partial orientation. The distinguished edge is unoriented and highlighted in grey. The signature has exactly one negative edge (with two opposite-pointing arrows).}
        \label{fig:fishFlow}
\end{figure}

    \medskip

    \noindent{Case 2: $C_k^*$ is negative ordinary}

    \medskip

    By Lemma \ref{lem:ordinary}, $C_k^*$ has no chords. Because every edge $e \in E(C_k^*)$ is negative, we may assume, by possibly reversing orientation of some edges, that every vertex is a source in $C_k^*$. We set $\phi(C_k^*) = \pm 1$, assigning the same value to each edge and choosing so that (by Lemma \ref{lem:ordinary}) the at most one vertex $w \in V(C_k^*)$ incident to an edge in $E_k$ has zero boundary (if it exists). Now we assign a value $\phi(f) = \pm 1$ to each edge $f \in F_k$ so that the vertices of degree three on $C_k^*$ have zero boundary. This leaves the vertex of degree two (if it exists) with non-zero boundary as desired.

\medskip

\noindent{Case 3: } $C_k^*$ is negative special

\medskip

For every edge $f$ that is either a chord of $C_k^*$ or in $F_k$, set $\phi(f) = 1$. Now by Lemma \ref{cycle_boundary1}, we may choose values $\phi(e) \in \{0,\pm 1\}$ for the edges in $V(C_k^*)$ so that each vertex of degree three in $V(C_k^*)$ has zero boundary and the unique vertex of degree two has non-zero boundary.

\medskip

\noindent{Case 4: $C_k^*$ is positive}

\medskip

We set $\phi(e) = 1$ for every chord edge $e$ of $C_k^*$. We proceed in two subcases.

If $|F_k| \ge 2$, then we proceed as in the analogous case in \cite{devos2023nowherezero}. That is, for each $f \in F(C_k)$ we choose $\phi(f) = \pm 1$ so that the sum of the boundary of $\phi$ on the vertices of $C_k^*$ is zero. This means that by Lemma \ref{cycle_boundary1} we may set $\phi(e) \in \{0,\pm 1\}$ for every edge $e \in E(C_k^*)$ so that $\phi$ has zero boundary at every vertex in $V(C_k^*)$.

Otherwise by Lemma \ref{lem:positive}, it must be that $|F_k| = 1$ and there exists $y \in V(C_k^*)$ so that $y$ is incident to an edge $e$ in $E_k$ and $e$ is a cut-edge in the graph $H'=G[U_k \cup V(C_k^*)]$. Let $H$ be the component of $H' - y$ containing the other end of $e$. Let $f$ be the (unique) edge in $F_k$. 

\begin{figure}[h]
    \centering
    \includegraphics[width=0.5\linewidth]{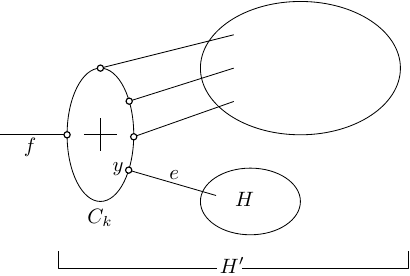}
    \caption{A positive cycle $C_k$ with $|F_k| = 1$.}
    \label{fig:pos}
\end{figure}

We must extend $\phi$ to $E(C_k^*) \cup \{f\}$. If $s = \sum_{v \in V(C_k^*)}\partial\phi(v) \neq 0$, then we may choose $\phi(f) = \pm 1$ to make the sum of $\partial \phi$ on the all the vertices of $C_k^*$ equal to zero. If $s=0$, we must first modify the preflow on the edges with an end in $V(H)$: for every edge $d$ in the domain of $\phi$, let $\phi'(d) = -\phi(d)$ if $d$ has an end in $V(H)$ and $\phi'(d) = \phi(d)$ otherwise. Then
$\partial \phi' (v) = -\partial \phi(v)$ for all vertices $v \in V(H) \cup \{y\}$, and $\partial \phi' (v)= \partial \phi(v)$ for the remaining vertices in $V(H')$.
This means $\phi'$ still satisfies the requirements of the preflow up to this step, but in particular $\partial \phi' (y) = -\partial \phi(y)$ (which is non-zero) while the boundaries of the remaining vertices in $V(C_k)$ haven't changed. This means $s' = \sum_{v \in V(C_k^*)}\partial\phi'(v) \neq 0$. But now letting $\phi = \phi'$ puts us in the $s \neq 0$ case and we choose $\phi(f) = \pm 1$ to make the sum of $\partial \phi$ on the all the vertices in $C_k^*$ equal to zero. Now by Lemma \ref{cycle_boundary1} we set $\phi(e) \in \{0,\pm 1\}$ on the edges of $C_k^*$ so that $\phi$ has zero boundary at every vertex in $V(C_k^*)$.
\end{proof}

\section{Proof of the main theorem}\label{3ecmain}

In this section we put together our lemmas to complete the proof of our main theorem. We follow the method set out in \cite{devos2023nowherezero}, and use machinery from there. We also call on the following theorem of Lu, Luo, and Zhang \cite{LLZ}.

\begin{theorem}[Lu et. al.]\label{thm:no2disjNeg}
If $G$ is a flow-admissible signed graph with no two edge-disjoint negative cycles, then $G$ has a nowhere-zero $6$-flow.
\end{theorem}

Next, we state two lemmas and a definition from \cite{devos2023nowherezero}, before we conclude with the proof of our main theorem.



\begin{lemma}[\cite{devos2023nowherezero}]
\label{perfect_matching}
Let $G$ be a subdivision of a 3-connected cubic graph, let $x_1, \ldots, x_k \in V_2(G)$, and for $1 \le i \le k$ let $C_i$ be a cycle of odd length in $G$ with $x_i \in V(C_i)$.  If the following conditions are satisfied:
\begin{itemize}
\item $|V_2(G)|$ is even,
\item $V(C_i) \cap V(C_j) = \emptyset$ whenever $i \neq j$,
\item $|V_2(G) \setminus \{x_1, \ldots, x_k\}| < 6$,
\item If $C_i$ straddles an edge cut $S$, then $|S| \neq 3$, and
\item If $C_i$ and $C_j$ straddle an edge cut $S$, and $i \neq j$, then $|S| \neq 4$,
\end{itemize}
then $G$ contains a perfect matching.
\end{lemma}


\begin{definition}[\cite{devos2023nowherezero}]\label{def:auxGraph}
    Let $\phi : E(G) \rightarrow \Z_3$ be a pre-flow in an oriented signed graph $G$ with all vertices of degree 2 or 3, and assume that $Z := \{ e \in E(G) \mid \phi(e) = 0 \}$ is a matching. We will show how to construct \textbf{\emph{the auxiliary graph associated with $G$ and $\phi$}}.
\end{definition}

Without loss of generality,  we may assume that every edge in $Z$ is positive, and $\phi(E(G)) \subseteq \{0,1\}$.  First, modify $G$ to form a new bidirected graph $G^+$ by adding a single new vertex $w$ and then for every vertex $v \in V_2(G)$ we add a positive edge with ends $v,w$.  We may extend the function $\phi$ to a function $\phi^+$ with domain $E(G^+)$ by assigning each newly added edge $e$ a suitable orientation and setting $\phi^+(e) = 1$ so that $\partial \phi^+(v) = 0$ holds for every $v \in V(G)$.  Let us comment that $w$ has only been added to make the graph (without $w$) cubic for the forthcoming modification and will be deleted later.  

\begin{figure}
    \centering
    \includegraphics[width=0.5\linewidth]{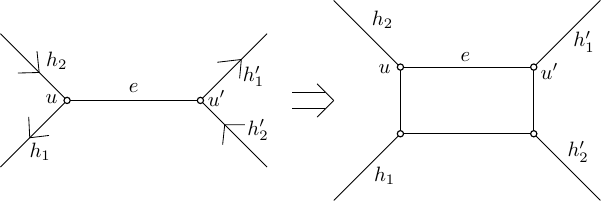}
    \caption{Auxiliary graph construction}
    \label{auxFig}
\end{figure}

Let $e = uu' \in E(G^+)$ have $\phi(e) = 0$. Note $w$ is not an end of $e$. Let $h_1, h_2$ and $h_1', h_2'$ be the other edges incident to $u$ and $u'$ respectively; without loss we may assume their nearest arrows to $u,u'$ are oriented as in the left side of Figure \ref{auxFig}. Modify $G^+$ by subdividing $h_1$ and $h_2'$, and adding a new edge in between the two new vertices. Let $H$ be the graph (not signed or oriented) obtained by doing this process at every edge where $\phi$ is zero, and then deleting the vertex $w$.  We call $H$ the \emph{auxiliary} graph \emph{associated with} $G$ and $\phi$. 

\begin{lemma}[\cite{devos2023nowherezero}]\label{4preflow}
Let $G$ be an oriented signed graph with all vertices of degree 2 or 3, let $\phi : E(G) \rightarrow {\mathbb Z}_3$ be a preflow with the property that $Z = \{ e \in E(G) \mid \phi(e) =0 \}$ is a matching.  If the auxiliary graph associated with $G$ and $\phi$ has a perfect matching, then there exists an integer preflow $\psi : E(G) \rightarrow \{0,\pm1,\pm2,\pm3\}$ so that $\psi + 3 \Z = \phi$, and $\partial \psi (v) = \pm 1$ for every $v \in V_2(G)$.
\end{lemma}

\medskip
\noindent Now we may conclude with our proof of the main result.

\begin{proof}[Proof of Theorem \ref{mainthm}]
Suppose for contradiction that the theorem is false. Let $G$ be a flow-admissible 3-edge-connected oriented signed graph which is a counterexample to the theorem and for which $3|E(G)| - 4|V(G)|$ is minimum. Then by Lemma \ref{getcubic}, $G$ is cubic; and  
by Observation~\ref{obs:well-behaved}, every proper induced subgraph of $G$ is well-behaved. Since $G$ is cubic, if $G$ does not have two vertex-disjoint negative cycles, then by Theorem \ref{thm:no2disjNeg} $G$ has a nowhere-zero 6-flow, which contradicts that $G$ is a counterexample. Thus it must be that $G$ has two vertex-disjoint negative cycles. 

Now, input $G$ to the Cycle Selection Algorithm, and let $\mathcal{C} = C_1, \dots, C_t$ be the list of cycles returned.
Modify $G$ to form $G^*$ by subdividing one edge from each $C_i$ that is either negative special or has even-length and is negative ordinary. Let $\mathcal{C}^* = C_1^*, \ldots, C_t^*$ be the list of cycles in $G^*$ corresponding to $\mathcal{C}$, and let the cycles in $\mathcal{C}^*$ inherit the type of their corresponding cycle. By Lemma \ref{lem:buildPreflow}, there is a preflow $\phi : E(G^*) \rightarrow {\mathbb Z}_3$ so that $\phi(e) = 0$ only when $e$ lies in a negative special cycle, a positive cycle, or $e$ is the distinguished edge in fish. Note that this means every vertex in $V(G^*)$ is incident to at least one edge $e$ with $\phi(e) \neq 0$. Because $G^*$ is subcubic, and by the boundary requirements of a preflow, the set $Z \in E(G^*)$ of edges $e$ with $\phi(e) = 0$ is therefore a matching. 

Let $H$ be the auxiliary graph associated with $G^*$ and $\phi$ (see Definition \ref{def:auxGraph}). We will show that each negative ordinary cycle in $C^*$ has a corresponding cycle in $H$ of odd-length. Let $e \in E(G^*)$ be an edge in a negative ordinary cycle in $C^*$. The structure of $phi$ means $e$ is not adjacent to an edge in $Z$. But every edge which is subdivided in the construction of $H$ is adjacent to an edge in $Z$. Hence it follows that $e$ is not subdivided in the construction of $H$. Thus every negative ordinary cycle in $C^*$ appears unaltered in $H$. Because the negative ordinary cycles in $G^*$ have odd-length, it follows that each has a corresponding cycle in $H$ which also has odd-length. We are particularly concerned with those cycles in $H$ which correspond to negative ordinary cycles of even-length in $\mathcal{C}$ (because they are subdivided in $C^*$ and so each contains a vertex of degree-two in $H$). Denote the set of those cycles by $\mathcal{C}_H$.

Let $\mathcal{N}$ be the set of all but the last two even-length negative ordinary cycles in $\mathcal{C}$, let $\mathcal{N}^*$ be the set of corresponding cycles in $G^*$, and let $\mathcal{N}_H$ be the set of corresponding cycles in $H$ (so $\mathcal{N}_H$ is all but the last two cycles in $\mathcal{C}_H$). By Lemma \ref{lem:negCycStraddle}, no cycle in $\mathcal{N}$ straddles a 3-edge-cut in $G$, and no two of these cycles straddle a $4$-edge-cut in $G$. By construction, it is straightforward to see that the same is true for $\mathcal{N^*}$ and $G^*$. Further, again by construction, the same is true for $H$ and $\mathcal{N}_H$: if there was an edge-cut $S$ of size 3 or 4 which violated this requirement, it would have to contain two of the edges in the 4-cycle of a gadget (see Figure \ref{auxFig}). But this is not possible because those edges are not in a cycle in $\mathcal{N}_H$, and all but at most one of the edges in $S$ are contained in some cycle in $\mathcal{N}_H$. 

Now, there are at most four degree-2 vertices in $H$ which are not contained in a cycle in $\mathcal{N}_H$ (the degree-2 vertices associated with the at most two negative special cycles in $\mathcal{C}$, and the degree-two vertices contained in a cycle in $\mathcal{C}_H$ but not $\mathcal{N}_H$). Also, by Lemma \ref{CYcSelLast} it follows that there are an even number of degree-two vertices in $H$. Hence, it follows from Lemma \ref{perfect_matching} that $H$ has a perfect matching. This means that by Lemma \ref{4preflow} (and the structure of $\phi$), there is a integer-preflow $\psi: E(G^*) \to \{0, \pm1, \pm2,\pm3\}$ so that $\psi(e) = 0$ only if $e$ lies on a positive cycle, a negative special cycle, or $e$ is the distinguished edge in fish; and $\partial \psi(v) = \pm 1$ for every $v \in V_2(G^*)$.

We construct $\Z$-preflow $\tau: E(G^*) \to \{0, \pm 1\} $ which will `match' $\phi$ at the vertices of degree-two (so that when $\phi, \tau$ are summed in a certain way, the result is a flow), and so that every edge in $E(G^*)$ is non-zero in at least one of $\phi, \tau$. 
Beginning with $\tau = 0$, we will modify $\tau$ so that it is nonzero on exactly the positive cycles, negative special cycles, negative ordinary cycles with a degree-two vertex, and a positive cycle in fish. For every $C_i^* \in \mathcal{C}^*$ that is positive, modify $\tau$ by assigning $\tau(e) = \pm 1$ for every $e \in E(C_i^*)$ so that $\partial \tau (v) = 0$ for every $v \in V(C_i^*)$. For every $C_i^* \in \mathcal{C}^*$ that is either negative special, or negative ordinary with a degree-two vertex, let $x_i$ be the degree-two vertex in $C_i^*$.  Modify $\tau$ by assigning $\tau(e) = \pm 1$ for every $e \in E(C_i^*)$ so that $\partial \tau(v) = 0$ for every $v \in V(C_i^*) \setminus \{x_i\}$ and so that $\partial \tau( x_i) = -2 \partial \psi(x_i)$.
Finally, if $C_t^* \in \mathcal{C}^*$ is type fish, let $C \subseteq C_t^*$ be a positive cycle containing the distinguished edge of $C_t^*$. Modify $\tau$ by assigning $\tau(e) = \pm 1$ for every $e \in E(C)$ so that $\partial \tau (v) = 0$ for every $v \in V(C)$. 

To conclude, notice that $2\psi + \tau$ is a nowhere-zero 8-flow in $G^*$. This gives a nowhere-zero $8$-flow in $G$, which contradicts that $G$ is a counterexample to the theorem and completes the proof.
\end{proof}

\bibliographystyle{acm}
\bibliography{bib}

\end{document}